\documentclass{amsart}

\usepackage{amssymb}
\usepackage{amsthm}
\usepackage{amsmath}
\usepackage{mathrsfs}			%for mathscr
\usepackage[usenames,dvipsnames]{color}
\usepackage[unicode]{hyperref}	%for hyperlinks
\hypersetup{urlcolor=blue, colorlinks=true}		%set color of links to blue
\usepackage{geometry}
\usepackage{subfigure}
\usepackage[bf]{caption}
\input xy
\xyoption{all}

\renewcommand{\phi}{\varphi}

\newcommand{\QQ}{\mathbf{Q}}
\newcommand{\ZZ}{\mathbf{Z}}
\newcommand{\RR}{\mathbf{R}}
\newcommand{\CC}{\mathbf{C}}

\newcommand{\ds}{\displaystyle}

\newcommand{\N}{\mc{N}}

\newcommand{\lr}{\langle\cdot,\!\cdot\rangle}
\newcommand{\ol}[1]{\overline{#1}}
\newcommand{\mf}[1]{\mathfrak{#1}}
\newcommand{\mc}[1]{\mathcal{#1}}
\newcommand{\wt}[1]{\widetilde{#1}}
\newcommand{\gl}{\mathrm{GL}}
\newcommand{\GO}{\mathrm{GO}}

\renewcommand{\mod}[1]{\text{\! }(\operatorname{mod}\text{\! }#1)}
\renewcommand{\varepsilon}{\epsilon}

\DeclareMathOperator{\tr}{tr}

\DeclareMathOperator{\SL}{SL}
\newcommand{\Hom}{\mathrm{Hom}}

\renewcommand{\Re}{\mathrm{Re}}
\renewcommand{\Im}{\mathrm{Im}}

\newcommand{\vect}[1]{\begin{pmatrix}#1\end{pmatrix}}

\newtheorem{theorem}{Theorem}[section]
\newtheorem{theoremA}{Theorem}
\newtheorem{proposition}[theorem]{Proposition}
\newtheorem{lemma}[theorem]{Lemma}

\newtheorem{remark}[theorem]{Remark}

\hypersetup{pdfauthor={Robert Harron},pdftitle={The shapes of pure cubic fields},pdfkeywords={Algebraic number theory, pure cubic fields, lattices, arithmetic statistics, equidistribution}}

\newcommand{\conorm}[7]{!{(-100,0)}*{p}}

\newcommand{\Rcom}[1]{}

\begin{document}

\title[The shapes of pure cubic fields]{T\MakeLowercase{he shapes of pure cubic fields}}

\author{R\MakeLowercase{obert} H\MakeLowercase{arron}}
\address{
Department of Mathematics\\
Keller Hall\\
University of Hawai`i at M\={a}noa\\
Honolulu, HI 96822\\
USA
}
\email{rharron@math.hawaii.edu}
\date{\today}

\begin{abstract}
We determine the shapes of pure cubic fields and show that they fall into two families based on whether the field is wildly or tamely ramified (of Type I or Type II in the sense of Dedekind). We show that the shapes of Type I fields are rectangular and that they are equidistributed, in a regularized sense, when ordered by discriminant, in the one-dimensional space of all rectangular lattices. We do the same for Type II fields, which are however no longer rectangular. We obtain as a corollary of the determination of these shapes that the shape of a pure cubic field is a complete invariant determining the field within the family of all cubic fields.
\end{abstract}

\maketitle

%++++++++++++++++++++++++++++++++++++++++++++++++++++++++++++++
%++++++++++++++++++++++++++++++++++++++++++++++++++++++++++++++
 
\tableofcontents
 
\section{Introduction}
The shape of a number field is an invariant coming from the geometry of numbers that can be seen as a refinement of the discriminant, or as a positive definite relative of the trace-zero form. Specifically, a degree $n$ number field $K$ can be embedded into its Minkowski space as $j_\RR:K\hookrightarrow K\otimes_\QQ\RR\cong\RR^n$ yielding a rank $n$ lattice $j_\RR(\mc{O}_K)\subseteq\RR^n$, where $\mc{O}_K$ denotes the ring of integers of $K$. The \emph{shape} of $K$ is defined to be the equivalence class of the rank $n-1$ lattice given by the image $\mc{O}_K^\perp$ of $j_\RR(\mc{O}_K)$ in $\RR^n/\RR j_\RR(1)$ up to scaling, rotations, and reflections. Explicitly, one obtains a quadratic form on $K$ via
\begin{equation}\label{eqn:quadform}
	\alpha\longmapsto\!\!\!\sum_{\sigma\in\Hom(K,\CC)}|\sigma(\alpha)|^2
\end{equation}
and the shape can be defined as the $\gl_{n-1}(\ZZ)$-equivalence class, up to scaling by positive real numbers, of the restriction of this form to the image of $\mc{O}_K$ under the projection map $\alpha\mapsto\alpha-\tr_{K/\QQ}(\alpha)/n$.
The shape of $K$ may equivalently be considered as an element of the double-coset space $\gl_{n-1}(\ZZ)\backslash\gl_{n-1}(\RR)/\GO_{n-1}(\RR)$, where $\GO$ denotes the group of orthogonal similitudes. This space, which we denote $\mc{L}_{n-1}$ and call the \emph{space of rank $(n-1)$ lattices}, carries a natural $\gl_{n-1}(\RR)$-invariant measure.

Little is known about the shapes of number fields. Their study was first taken up in the PhD thesis of David Terr \cite{Terr}, a student of Hendrik Lenstra's. In it, Terr shows that the shapes of both real and complex cubic fields are equidistributed in $\mc{L}_2$ when the fields are ordered by the absolute value of their discriminants. This result has been generalized to $S_4$-quartic fields and to quintic fields by Manjul Bhargava and Piper Harron in \cite{Manjul-Piper, PiperThesis} (where the cubic case is treated, as well). Aside from this, Bhargava and Ari Shnidman \cite{Manjul-Ari} have studied which real cubic fields have a given shape, and in particular which shapes are possible for fields of given quadratic resolvent. Finally, Guillermo Mantilla-Soler and Marina Monsurr\`{o} \cite{M-SM} have determined the shapes of cyclic Galois extensions of $\QQ$ of prime degree. Note however that \cite{M-SM} uses a slightly different notion of shape: they instead restrict the quadratic form in \eqref{eqn:quadform} to the space of elements of trace zero in $\mc{O}_K$. It is possible to carry out their work with our definition and our work with theirs; the answers are slightly different as described below.\footnote{Also, when \cite{Manjul-Ari} refers to `shape' in the complex cubic case, they use the trace form instead of \eqref{eqn:quadform}, yielding an indefinite quadratic form, hence the need for the present article.}

This article grew out of the author's desire to explore an observation he made that, given a fixed quadratic resolvent in which $3$ ramifies, the shapes of real $S_3$-cubic fields sort themselves into two sets depending on whether $3$ is tamely or wildly ramified; a \emph{tame versus wild dichotomy}, if you will. Considering complex cubics, the pure cubics---that is, those of the form $\QQ(m^{1/3})$---are in many ways the simplest. These were partitioned into two sets by Dedekind, who, in the typical flourish of those days, called them Type I and Type II. In our context, \emph{pure} cubics are exactly those whose quadratic resolvent is $\QQ(\omega)$, $\omega$ being a primitive cube root of unity, and Type I (resp.\ Type~II) corresponds to $3$ being wildly (resp.\ tamely) ramified. The first theorem we prove (Theorem~\ref{thm:A}) computes the shape of a given pure cubic field and shows that, just as in the real case, pure cubic shapes exhibit a tame versus wild dichotomy.\footnote{We remark that this dichotomy also appears in \cite{M-SM} and in ongoing work of the author with Melanie Matchett Wood. In the former, it is shown that, for a given prime $\ell\geq5$, only two `shapes' (in the sense of \cite{M-SM}) are possible for cyclic extensions of degree $\ell$, and the `shape' is exactly determined by whether $\ell$ is tame or wild (for $\ell=3$, the two possibilities collapse to one, a result already appearing in \cite{Terr}, and for $\ell=2$, well, there is only one rank one lattice). We however remark that, with our definition of shape, the dichotomy in \cite{M-SM} disappears, while with their definition, \textit{all} pure cubics have rectangular `shape'. It would be very interesting to better understand the subtle difference between the two definitions. In the ongoing project with Wood, the tame versus wild dichotomy is shown to hold for totally real Galois quartic fields.} In the real cubic case, \cite{Manjul-Ari} shows that for fields of a fixed quadratic resolvent, there are only finitely many options for the shape. However, there are infinitely many possibilities for the shape of a pure cubic: indeed, Theorem~\ref{thm:B} shows that there is a bijection between pure cubic fields and their shapes! In fact, this theorem goes further and says that the shape of a pure cubic field uniquely determines it (up to isomorphism) within the collection of all cubic fields. Contrast this with Mantilla-Soler's result that says that complex cubic fields of the same discriminant have isometric integral trace forms \cite[Theorem~3.3]{M-S}.

A natural question to ask is: how are the shapes distributed? The results of \cite{Manjul-Ari}, together with the aforementioned observation, imply that, after splitting the shapes of real $S_3$-cubic fields as needed according to tame versus wild, the shapes are equidistributed amongst the finitely many options. In Theorem~\ref{thm:C}, we show that the shapes of Type~I and Type~II pure cubic fields are equidistributed within their respective one-dimensional family of shapes (viewing $\mc{L}_2$ as a subset of the upper-half plane, and noting that the natural measure on $\mc{L}_2$ is given by the hyperbolic metric on the upper-half plane). Unlike the other equidistribution results described above, the spaces of lattices under consideration have infinite measure. We are thus lead to a ``regularized'' notion of equidistribution taking into account the different growth rates that occur for fields whose shape lies in bounded versus unbounded subsets of the one-dimensional spaces of lattices.

\subsection{Statement of the main theorems}
Let $K=\QQ(m^{1/3})$ be a pure cubic field and write $m=ab^2$, where $a$ and $b$ are relatively prime, squarefree, positive integers. Note that the cube root of $m^\prime:=a^2b$ also generates $K$ (as $mm^\prime$ is a cube), so we may, and do, assume that $a>b$. We will refer to $r_K:=a/b$ as the \emph{ratio of $K$}. Then, $K$ is Type II if and only if $3\nmid m$ and $r_K\equiv\pm1\mod{9}$.

As described above, the shape of $K$ is a rank 2 lattice, up to orthogonal similitudes. For a lattice of rank $2$, it is common to consider it as living in $\CC$ and to take a representative of its equivalence class that has basis $\{1,z\}$ with $z\in\mf{H}$, the upper-half plane. Changing the basis by an element of $\SL(2,\ZZ)$ and rewriting the basis as $\{1,z^\prime\}$ corresponds to the action of $\SL(2,\ZZ)$ by fractional linear transformations on $\mf{H}$. In fact, we allow changing the basis by $\gl(2,\ZZ)$: defining
\[
	\vect{0&1\\1&0}\cdot z:=1/\ol{z}
\]
allows us to extend the action of $\SL(2,\ZZ)$ to all of $\gl(2,\ZZ)=\SL(2,\ZZ)\sqcup\vect{0&1\\1&0}\SL(2,\ZZ)$. Thus, the space of rank 2 lattices may be identified (up to some identification on the boundary) with a fundamental domain for the action of $\gl(2,\ZZ)$ on $\mf{H}$ given by
\[
	\mc{F}:=\{z=x+iy\in\mf{H}:0\leq x\leq1/2, x^2+y^2\geq1\}.
\]

Here is our first main theorem.
\begin{theoremA}\label{thm:A}\mbox{}
	\begin{itemize}
		\item Pure cubic fields of Type I have shapes lying on the imaginary axis inside $\mc{F}$; specifically, the shape of $K$ is $ir_K^{1/3}\in\mc{F}$. Thus, their shapes are rectangular.
		\item Pure cubic fields of Type II have shapes lying on the line segment $\Re(z)=1/3$, $\Im(z)>1/3$ in $\mf{H}$; specifically, the shape of $K$ is $(1+ir_K^{1/3})/3\in\mf{H}$. The shapes of Type II pure cubics are thus parallelograms with no extra symmetry.
	\end{itemize}
\end{theoremA}
This implies that pure cubic fields do indeed exhibit a tame versus wild dichotomy. The proof of this theorem will be accomplished by explicitly writing down a basis for a lattice representing the shape.

\begin{remark}
	For Type II fields, the given line segment does not lie in $\mc{F}$; instead, it lies in the union of $\mc{F}$ with its translates by $\gamma_1=W$, $\gamma_2=SU$, and $\gamma_3=SUSW$, where
	\[
		W=\vect{0&1\\1&0},\quad S=\vect{0&-1\\1&0},\quad\text{and}\quad U=\vect{1&-1\\0&1}.
	\]
	These 4 pieces correspond to $r_K^{1/3}$ in the intervals $I_0=(\sqrt{8},\infty)$, $I_1=(\sqrt{5},\sqrt{8})$, $I_2=(\sqrt{2},\sqrt{5})$, and $I_3=(1,\sqrt{2})$, respectively. To describe the location of the shapes within $\mc{F}$, one need simply act on the appropriate subsegments of the line segment by $\gamma_i^{-1}$. One then obtains the union of a line segment and three circular arcs: (the intersection of $\mc{F}$ with) the line $\Re(z)=1/3$ and the circles centred at $3/2$, $-1/2$, and $1/2$, all of radius $3/2$.
\end{remark}

An important corollary to this theorem is that the shape provides a complete invariant for pure cubic fields.
\begin{theoremA}\label{thm:B}
	If the shape of a cubic field is one of the shapes occurring in Theorem~\ref{thm:A}, then the field is the uniquely determined pure cubic field of that shape.
\end{theoremA}
Once we know Theorem~\ref{thm:A}, the proof of this result is a simple argument involving the field of rationality of the coordinates of the shape in $\mf{H}$. Note that there are plenty of pure cubic fields sharing the same discriminant. For instance, $\QQ(6^{1/3})$ and $\QQ(12^{1/3})$ both have discriminant $-3^3\cdot 6^2$, but their shapes are $i6^{1/3}$ and $i(\frac{3}{2})^{1/3}$, respectively. This result illustrates how different the shape and the trace form can be. Indeed, Mantilla-Soler shows in \cite[Theorem~3.3]{M-S} that two complex cubic fields have isometric integral trace form if and only if they have the same discriminant. It is natural to ask whether the shape is a complete invariant for complex cubic fields in general.\footnote{The shape fails miserably to distinguish (real) Galois cubic fields: they all have hexagonal shape!} We will present a positive answer to this question in upcoming work using a rationality argument as in the proof of Theorem~\ref{thm:B} below.

As for equidistribution, we must introduce measures on $\mathscr{S}_{\mathrm{I}}:=\{iy:y\geq1\}$ and $\mathscr{S}_{\mathrm{II}}:={\{\frac{1+iy}{3}:y\geq1\}}$. As a natural choice, we pick those induced by the hyperbolic metric on $\mf{H}$, equivalently those invariant under a subgroup of $\SL_2(\RR)$. Recall the hyperbolic line element on $\mf{H}$ is
\[
	ds=\frac{\sqrt{dx^2+dy^2}}{y}
\]
and thus induces the measure $\frac{dy}{y}$ on both $\mathscr{S}_{\mathrm{I}}$ and $\mathscr{S}_{\mathrm{II}}$. Alternatively, note that the diagonal torus $\{\operatorname{diag}(y,y^{-1})\}\subseteq\SL_2(\RR)$ acting on $i$ yields a homeomorphism from itself to the upper imaginary axis mapping the elements with $y>1$ onto $\mathscr{S}_{\mathrm{I}}$. As such, it induces an invariant measure on $\mathscr{S}_{\mathrm{I}}$, namely $\frac{dy}{y}$.\footnote{To be precise, $\operatorname{diag}(y,y^{-1})$ sends $i$ to $y^2i$, so in terms of the coordinate $iy$, the induced measure is $\frac{dy}{2y}$. We will ignore this $1/2$.} By conjugating this torus, we obtain the same for $\mathscr{S}_{\mathrm{II}}$. We denote the induced measures on $\mathscr{S}_{\textrm{I}}$ and $\mathscr{S}_{\textrm{II}}$ by $\mu_{\textrm{I}}$ and $\mu_{\textrm{II}}$, respectively. A slight complication arises as equidistribution is typically considered for finite measures, whereas on both $\mathscr{S}_{\mathrm{I}}$ and $\mathscr{S}_{\mathrm{II}}$ the measure $\frac{dy}{y}$ is merely $\sigma$-finite. This is in fact reflected in the asymptotics for pure cubic fields! Indeed, the number of pure cubic fields of discriminant bounded by $X$ is on the order of $\sqrt{X}\log(X)$ (see \cite[Theorem~1.1]{Cohen-Morra} or \cite[Theorem~8]{Manjul-Ari}), whereas those of bounded discriminant with shape lying in a finite interval (i.e.\ with $r_K$ in a finite interval) only grow like $\sqrt{X}$ (see Theorem~\ref{thm:fieldasymptotics} below). We thus ``regularize'' our notion of equidistribution and obtain the following result.
\begin{theoremA}\label{thm:C}
Define
\[
	C_{\mathrm{I}}=\frac{2C\sqrt{3}}{15}\quad\text{and}\quad C_{\mathrm{II}}=\frac{C\sqrt{3}}{10},
\]
where
\[
	C=\prod_p\left(1-\frac{3}{p^2}+\frac{2}{p^3}\right),
\]
the product being over all primes $p$. For $?=$ I, resp.~II,  and real numbers $1\leq R_1<R_2$, let $[R_1,R_2)_?$ denote the ``interval'' $i[R_1,R_2)$, resp.~$(1+i[R_1,R_2))/3$, in $\mathscr{S}_?$.
Then, for all $R_1, R_2$,
\[
	\lim_{X\rightarrow\infty}\frac{\#\left\{K\text{ of type ?}:|\Delta(K)|\leq X,\operatorname{sh}(K)\in[R_1,R_2)_?\right\}}{C_?\sqrt{X}}=\int_{[R_1,R_2)_?}d\mu_?,
\]
where $\Delta(K)$ is the discriminant of $K$ and $\operatorname{sh}(K)$ is the shape of $K$ (taken in $\mathscr{S}_?$).
\end{theoremA}

\begin{remark}\mbox{}
	\begin{enumerate}
		\item For the usual, unregularized, notion of equidistribution, the left-hand side above would be
		\[
			\lim_{X\rightarrow\infty}\frac{\#\left\{K:|\Delta(K)|\leq X,\operatorname{sh}(K)\in[R_1,R_2)_?\right\}}{\#\left\{K:|\Delta(K)|\leq X\right\}},
		\]
		where the denominator is $C_?\sqrt{X}\log(X)+o(\sqrt{X})$. Given the different growth rates of the numerator and the denominator, this limit is always $0$.
		\item We phrase this theorem in a more measure-theoretic way below (see Theorem~\ref{thm:C'}), defining a sequence of measures $\mu_{?,X}$ that converges weakly to $\mu_?$.
	\end{enumerate}
\end{remark}

\section{Determining the shape}\label{sec:proofofthmA}
In this section, we provide integral bases for pure cubic fields (\S\ref{sec:int_bases}) and go on to use these bases to explicitly determine bases of $\mc{O}_K^\perp$ (\S\ref{sec:shape}). These calculations directly prove Theorem~\ref{thm:A}. In \S\ref{sec:thumb}, we prove Theorem~\ref{thm:B}.

\subsection{Some basic facts about pure cubic fields}\label{sec:int_bases}
We briefly determine an integral basis and the discriminant of a pure cubic field, and explain how ramification at $3$ is what distinguishes Type~I from Type~II.

Every pure cubic field $K$ corresponds to exactly one pair $(m,m^\prime)$, where $m=ab^2$, $m^\prime=a^2b$, and $a$ and $b$ are two positive, relatively prime, squarefree integers whose product is at least $2$. Let $\alpha$ and $\beta$ be the roots in $K$ of $x^3-m$ and $x^3-m^\prime$, respectively, so that $\beta=\alpha^2/b$, $\alpha=\beta^2/a$, and $K=\QQ(\alpha)=\QQ(\beta)$. The discriminant of $x^3-m$ is $-3^3a^2b^4$, so that $\Delta(K)\mid\gcd(-3^3a^2b^4, 3^3a^4b^2)=-3^3a^2b^2$. The fact that $x^3-m$ (resp.\ $x^3-m^\prime$) is Eisenstein at all primes dividing $a$ (resp.\ $b$) implies that the index $[\mathcal{O}_K:\ZZ[\alpha,\beta]]$ is relatively prime to $ab$, and hence divides $3$. Thus, if $3|m$, then the index is $1$ and $\{1,\alpha,\beta\}$ is an integral basis of $\mc{O}_K$. Otherwise, since $[\mathcal{O}_K:\ZZ[\alpha,\beta]]=[\mathcal{O}_K:\ZZ[\alpha-m,\beta]]$, we may consider the minimal polynomial $(x+m)^3-m$ of $\alpha-m$. It is Eisenstein at $3$ if and only if $m\not\equiv\pm1\mod{9}$,\footnote{In fact, it is Eisenstein if and only if $m^3\not\equiv m\mod{9}$, but the only cubes modulo $9$ are $0,\pm1$, and $3\nmid m$, by assumption. Note also that the condition $m\equiv\pm1\mod{9}$ is equivalent to $r_K\equiv\pm1\mod{9}$; indeed, $r_K=m/b^3$.} in which case, again, $\{1,\alpha,\beta\}$ is an integral basis of $\mc{O}_K$. For $m\equiv\pm1\mod{9}$,  (the ``Type II'' fields) one may check by hand that the element $\nu=(1\pm\alpha+\alpha^2)/3$ is integral. Clearly, $[\ZZ[\nu,\beta]:\ZZ[\alpha,\beta]]=3$, so $\{1,\nu,\beta\}$ is then an integral basis of $\mc{O}_K$. This shows that the discriminant of $K$ is $-3^3a^2b^2$ (resp.\ $-3a^2b^2$) for Type I (resp.\ Type II). We summarize these results in the following lemma.

\begin{lemma}\label{lem:intbasis}
	With the above notation,
	\begin{itemize}
		\item the discriminant of a pure cubic of Type I is $\Delta(K)=-3^3a^2b^2$ and $\{1,\alpha,\beta\}$ is an integral basis;
		\item the discriminant of a pure cubic of Type II is $\Delta(K)=-3a^2b^2$ and $\{1,\nu,\beta\}$ is an integral basis.
	\end{itemize}
\end{lemma}

The two possible factorizations of a ramified prime in a cubic field are $\mf{p}_1\mf{p}_2^2$ and $\mf{p}^3$. The prime $3$ being tamely ramified means that $3=\mf{p}_1\mf{p}_2^2$, where $\mf{p}_1$ is unramified and $\mf{p}_2$ is tamely ramified. Consequently, $\mf{p}_2=\mf{p}_2^{2-1}$ exactly divides the different of $K/\QQ$. Since $\mf{p}_2$ has norm $3$, this implies $3$ exactly divides $\Delta(K)$. On the other hand, if $3$ is wild, then $3=\mf{p}^3$, in which case $\mf{p}^3$ divides the different. Again, $\mf{p}$ has norm $3$ so $3^3$ divides $\Delta(K)$.

\subsection{The shape}\label{sec:shape}
Let $\omega:=e^{2\pi i/3}\in\CC$ and let $m^{1/3}$ denote the real cube root of $m$. Let $j_\RR:K\rightarrow\RR^3$ denote the embedding into the ``Minkowski space'' of $K$ (following Neukirch's normalizations, see \cite[\S{I.5}]{Neukirch}): if $\sigma$ denotes the unique real embedding and $\tau$ denotes the complex embedding sending $\alpha$ to $\omega m^{1/3}$, then
\[
	j_\RR(a):=\big(\sigma(a),\Re(\tau(a)),\Im(\tau(a))\big).
\]
We also let $j:K\rightarrow\CC^3$ be given by $j(a)=(\sigma(a),\tau(a),\ol{\tau}(a))$. It is understood that the space $\RR^3$ is equipped with the inner product $\lr$ given by the diagonal matrix $\operatorname{diag}(1,2,2)$. This is the inner product obtained by restricting the standard Hermitian inner product on $\CC^3$ to the image of $j$, identified with $\RR^3$ as in \cite[\S{I.5}]{Neukirch}. As it is unlikely to cause confusion, we shall abuse notation and write $\alpha$ and $\beta$ to sometimes mean $\sigma(\alpha)=m^{1/3}$ and $\sigma(\beta)=m^{\prime1/3}$ depending on context.

Given a rank 2 lattice $\Lambda$ with Gram matrix\footnote{Recall that a Gram matrix for a lattice $\Lambda$ is the symmetric matrix whose $(i,j)$-entry is the inner product of the $i$th and $j$th basis vectors, for some choice of basis.} $B=(B_{ij})$ the associated point in $\mf{H}$ is $z=x+iy$, where
\begin{equation}\label{eqn:xycoords}
	x=\frac{B_{0,1}}{B_{0,0}},\quad y=\sqrt{\frac{B_{1,1}}{B_{0,0}}-x^2}.
\end{equation}

The starting point is the following proposition which immediately proves Theorem~\ref{thm:A} for Type I fields.

\begin{proposition}\label{prop:shape_of_OK}
	For any pure cubic field $K$, the image of the set $\{1,\alpha,\beta\}$ under $j_\RR$ is an orthogonal set whose Gram matrix is
	\[	\vect{3\\& 3\alpha^2\\&&3\beta^2}.
	\]
\end{proposition}
\begin{proof}
	First off,
	\[
		j(1)=(1,1,1),\quad j(\alpha)=(\alpha,\omega\alpha,\omega^2\alpha),\quad\text{and}\quad j(\beta)=(\beta,\omega^2\beta,\omega\beta).
	\]
	Thus, the proposition comes down to computing the Hermitian dot products of these vectors.
\end{proof}

This yields Theorem~\ref{thm:A} for Type I fields as follows. The set $\{j_\RR(\alpha),j_\RR(\beta)\}$ is a basis of $\mc{O}_K^\perp$ whose Gram matrix is
\[
	\vect{3\alpha^2\\&3\beta^2}.
\]
Using the conversion to $z=x+iy\in\mf{H}$ in Equation~\eqref{eqn:xycoords} gives $x=0$ and $y=(\beta/\alpha)=(a/b)^{1/3}=r_K^{1/3}$.

To study Type II fields, we will need to first know the image of $\nu$ in $\mc{O}_K^\perp$.
\begin{lemma}
	For $a\in K$, the image $a^\perp$ of $a$ in $\mc{O}_K^\perp$ is
	\[
		a^\perp=j_\RR(a)-\frac{\langle j_\RR(a),j_\RR(1)\rangle}{3}j_\RR(1).
	\]
\end{lemma}
This follows from linear algebra and the simple fact that $\langle j_\RR(1),j_\RR(1)\rangle=3$.

By Proposition~\ref{prop:shape_of_OK}, $\alpha^\perp=j_\RR(\alpha)$ and $\beta^\perp=j_\RR(\beta)$. Since $j_\RR(\nu)=(j_\RR(1)\pm j_\RR(\alpha)+bj_\RR(\beta))/3$, we get that
\begin{equation}
	\nu^\perp=j_\RR(\nu)-\frac{1}{3}j_\RR(1)=\frac{1}{3}(\pm\alpha^\perp+b\beta^\perp).
\end{equation}

\begin{lemma}
	The Gram matrix of the basis $\{\nu^\perp,\beta^\perp\}$ of $\mc{O}_K^\perp$ is
	\[
		\vect{\ds\frac{\alpha^2(1+\alpha^2)}{3} & \ds\frac{\alpha^4}{b} \\ \ds\frac{\alpha^4}{b} & \ds\frac{3\alpha^4}{b^2}}.
	\]
\end{lemma}
\begin{proof}
	Since $\langle\alpha^\perp,\beta^\perp\rangle=0$, we have
	\[
		\langle\nu^\perp,\nu^\perp\rangle=\frac{1}{9}(\langle\alpha^\perp,\alpha^\perp\rangle+b^2\langle\beta^\perp,\beta^\perp\rangle)=\frac{1}{9}(3\alpha^2+3b^2\beta^2).
	\]
	This is the claimed value since $b\beta=\alpha^2$. Furthermore,
	\[
		\langle\nu^\perp,\beta^\perp\rangle=\frac{b}{3}\langle\beta^\perp,\beta^\perp\rangle=b\beta^2=\alpha^4/b.
	\]
	Finally, we already know the inner product of $\beta^\perp$ with itself.	
\end{proof}
This basis is not terribly pleasing from the point of view of shapes. Making the following change of basis yields a nicer Gram matrix and proves Theorem~\ref{thm:A} for Type II fields.
\begin{lemma}
	Let $\epsilon\in\{\pm1\}$ and $k\in\ZZ$ be such that $b=3k+\epsilon$. Define
	\[
		\gamma=\vect{3&-b\\1&-k}.
	\]
	The basis $\{v_1,v_2\}$ of $\mc{O}_K^\perp$ given by
	\[
		\vect{v_1\\v_2}=\gamma\vect{\nu^\perp\\\beta^\perp}
	\]
	has the Gram matrix
	\[
		\vect{3\alpha^2 & \alpha^2 \\ \alpha^2 & \ds\frac{\alpha^2+\beta^2}{3}}.
	\]
	The associated point in $\mf{H}$ is $\frac{1}{3}+i\frac{r_K^{1/3}}{3}$.
\end{lemma}
\begin{proof}
	We have that
	\[
		v_1=3\nu^\perp-b\beta^\perp=\pm\alpha^\perp+\alpha^{\perp2}-\alpha^{\perp2}=\pm\alpha^\perp
	\]
	and
	\[
		v_2=\nu^\perp-k\beta^\perp=\frac{\pm\alpha^\perp+(3k+\epsilon)\beta^\perp-3k\beta^\perp}{3}=\frac{1}{3}(\pm\alpha^\perp+\epsilon\beta^\perp).
	\]
	Thus,
	\begin{align*}
		\langle v_1,v_1\rangle=3\alpha^2,\quad\langle v_1,v_2\rangle=\frac{1}{3}\langle\alpha^\perp,\alpha^\perp\rangle=\alpha^2,\quad\\
		\langle v_2,v_2\rangle=\frac{1}{9}(\langle\alpha^\perp,\alpha^\perp\rangle+\langle\beta^\perp,\beta^\perp\rangle)=\frac{\alpha^2+\beta^2}{3}.
	\end{align*}
	The associated point in $\mf{H}$ has $x=\alpha^2/(3\alpha^2)=1/3$ and
	\[
		y=\sqrt{\frac{\alpha^2+\beta^2}{9\alpha^2}-\frac{1}{9}}=\frac{\beta}{3\alpha}=\frac{r_K^{1/3}}{3}.
	\]
\end{proof}

This concludes the proof of Theorem~\ref{thm:A}.

\subsection{Proof of Theorem~\ref{thm:B}}\label{sec:thumb} Let $L$ be any `non-pure' cubic field and let $z_L=x_L+iy_L$ be any point in the upper-half plane corresponding to the lattice $\mc{O}_L^\perp$. By studying the fields of definition of $x_L$ and $y_L$, we will briefly show that $z_L$ cannot be the shape of a pure cubic field. Since Theorem~\ref{thm:A} already shows that non-isomorphic pure cubic fields have different shapes, Theorem~\ref{thm:B} will follow.

First, we note that the key result from Theorem~\ref{thm:A} needed here is that the $y$-coordinate of the shape of a pure cubic field $K$ is \textit{not} rational, rather it lies `properly' in the image of $K$ in $\RR$. For this paragraph, we will call such a number `purely cubic'. Now, if $L$ is a real cubic field, then the inner product on $L\otimes_\QQ\RR$ is simply the trace form and so the Gram matrix of $\mc{O}_L^\perp$ has coefficients in $\QQ$. The equations in \eqref{eqn:xycoords} giving $z_L$ in terms of the Gram matrix show that $x_L\in\QQ$ and $y_L$ is in some quadratic extension of $\QQ$. Thus, $y_L$ is not purely cubic. For a complex cubic field $L$, let $\sigma_L$ denote the unique real embedding of $L$. By definition of the inner product on $L\otimes_\QQ\RR$ as recalled above, the entries of the Gram matrix of any basis of $\mc{O}_L^\perp$ lie in the Galois closure $N_L$ of $\sigma_L(L)$. The equations  in \eqref{eqn:xycoords} then show that $x_L\in N_L$ and that $y_L$ is in a quadratic extension $\wt{N}_L$ of $N_L$. Since cubic extensions have no non-trivial intermediate extensions, for any pure cubic field $K$, we have that $\wt{N}_L\cap K=\QQ$. Thus, again, $y_L$ is not purely cubic.

\section{Equidistribution of shapes}\label{sec:equid}
We have organized this section beginning with the most conceptual and ending with the most detailed. Specifically, \S\ref{sec:carefree} contains the real mathematical content: counting ``strongly carefree couples'' in a cone below a hyperbola, subject to congruence conditions. In \S\ref{sec:fieldcounting}, these results are translated into counts for pure cubic fields of bounded discriminant with ratio in a given interval. Finally, to begin with, \S\ref{sec:proofThmC} consists mostly of some basic measure theory in order to introduce and prove Theorem~\ref{thm:C} in a conceptually nicer way. 

\subsection{Proof of Theorem~\ref{thm:C}}\label{sec:proofThmC}
For $?=$ I or II, let $C_c(\mathscr{S}_?)$ denote the space of continuous functions on $\mathscr{S}_?$ with compact support, and for a positive integer $X$, define a positive linear functional $\phi_{?,X}$ on $C_c(\mathscr{S}_?)$ by
\[
	\phi_{?,X}(f)=\frac{1}{C_?\sqrt{X}}\sum_{\substack{K\text{ of type ?}\\ |\Delta(K)|\leq X}}f(\operatorname{sh}(K)),
\]
where the shape of $K$ is taken in $\mathscr{S}_?$. By the Riesz representation theorem, each $\phi_{?,X}$ corresponds to a (regular Radon) measure $\mu_{?,X}$ on $\mathscr{S}_?$. Note that since these measures are finite sums of point measures, we have that
\[
	\mu_{?,X}([R_1,R_2)_?)=\frac{\#\left\{K\text{ of type ?}:|\Delta(K)|\leq X,\operatorname{sh}(K)\in[R_1,R_2)_?\right\}}{C_?\sqrt{X}},
\]
for all $1\leq R_1<R_2$.
\begin{theorem}\label{thm:C'}
For $?=$ I or II, the sequence $\mu_{?,X}$ converges weakly to $\mu_?$. That is, for all $f\in C_c(\mathscr{S}_?)$,
\[
	\lim_{X\rightarrow\infty}\int fd\mu_{?,X}=\int fd\mu_?.
\]
\end{theorem}
\begin{proof}
	The first step will be to prove the statement given in Theorem~\ref{thm:C} in the introduction, i.e.\ we must obtain asymptotics for the number of pure cubic fields of bounded discriminant and ratio in a given interval. This is the subject of Theorem~\ref{thm:fieldasymptotics} below. We obtain, in the notation of that theorem,
	\[
		\mu_{?,X}([R_1,R_2)_?)=\frac{\mc{N}_?(X,R_1^3,R_2^3)}{C_?\sqrt{X}}=\log\left(\frac{R_2}{R_1}\right)+o(1).
	\]
	Thus, $\mu_{?,X}([R_1,R_2)_?)\rightarrow\mu_?([R_1,R_2)_?)$ as desired. Passing from this result to the full result is a standard argument. Indeed, let $f\in C_c(\mathscr{S}_?)$ be given and let $\epsilon>0$ be arbitrary. Then, there exists a pair $(f_1,f_2)$ of step functions (i.e.\ finite linear combinations of characteristic functions of intervals of the form $[R_1,R_2)_?$ above) with $f_1\leq f\leq f_2$ and $\int(f_2-f_1)d\mu_?<\epsilon$. Since we already know the theorem for such step functions, we obtain that
	\begin{align*}
		\int fd\mu_?-\epsilon&\leq\int f_1d\mu_?=\lim_{X\rightarrow\infty}\int f_1d\mu_{?,X}\\
		&\leq\liminf_{X\rightarrow\infty}\int fd\mu_{?,X}\\
		&\leq\limsup_{X\rightarrow\infty}\int fd\mu_{?,X}\\
		&\leq\lim_{X\rightarrow\infty}\int f_2d\mu_{?,X}=\int f_2d\mu_?\\
		&\leq\int fd\mu_?+\epsilon.
	\end{align*}
	As $\epsilon$ is arbitrary, the result follows.
\end{proof}
%Recall the subsets $\mathscr{S}_{\mathrm{I}}$ and $\mathscr{S}_{\mathrm{II}}$ of $\mf{H}$ from the introduction.
%Version 1
%We use \cite[\S\S1.4,1.5]{Folland} as a basic reference for this section.
%
%For $?=$ I or II, let $\mathscr{A}_?$ denote the algebra of sets in $\mathscr{S}_?$ consisting of finite disjoint unions of intervals of the form $[R_1, R_2)_?$ or $[R,\infty)_?$. For $A\in\mathscr{A}_?$, let
%\[
%	\wt{\mu}_?(A)=\lim_{X\rightarrow\infty}\frac{\#\left\{K:|\Delta(K)|\leq X,\operatorname{sh}(K)\in A\right\}}{C_?\sqrt{X}}.
%\]
%It is clear that this is a premeasure. Since $\wt{\mu}_?$ is $\sigma$-finite, there is a unique Borel measure $\mu_?$ on $\mathscr{S}_?$ extending $\wt{\mu}_?$. We may now state our strengthening of Theorem~\ref{thm:C}.
%\begin{theorem}\label{thm:C'}
%For $?=$ I or II,
%\[
%	\mu_?=\frac{dy}{y}.
%\]
%\end{theorem}
%\begin{proof}
%	By Theorem~\ref{thm:fieldasymptotics} below, the measures $\mu_?$ and $\frac{dy}{y}$ agree for any finite interval $[R_1,R_2)_?$; indeed, they both give measure $\log(R_2/R_1)$. Furthermore, since the number of pure cubic fields of Type ? and discriminant bounded by $X$ grows like $\sqrt{X}\log(X)$ (see \cite[Theorem~8]{Manjul-Ari}), we also see that $\mu_?([R,\infty))=\infty$, in agreement with the measure $\frac{dy}{y}$. As such, the measures agree on all of $\mathscr{A}_?$. We obtain the desired result by the uniqueness of the extension of $\wt{\mu}_?$.
%\end{proof}

\subsection{Counting pure cubic fields of bounded discriminant and ratio}\label{sec:fieldcounting}
For real numbers $X, R_1,R_2\geq1$ with $R_1<R_2$, let $\N_?(X,R_1,R_2)$ denote the number of pure cubic fields (up to isomorphism) of type $?=$~I or II, of discriminant less than $X$ (in absolute value), and ratio in the interval $(R_1,R_2)$. When we drop the subscript, we count Type I and II together.

\begin{theorem}\label{thm:fieldasymptotics}
	For fixed $R_1$ and $R_2$, we have the following asymptotics:
	\begin{align}
		\N_{\mathrm{I}}(X,R_1,R_2)&=\frac{2C}{15\sqrt{3}}\sqrt{X}\log\frac{R_2}{R_1}+o(\sqrt{X}),\\
		\N_{\mathrm{II}}(X,R_1,R_2)&=\frac{C}{10\sqrt{3}}\sqrt{X}\log\frac{R_2}{R_1}+o(\sqrt{X}),
\end{align}
\begin{align}
		\N(X,R_1,R_2)&=\frac{7C}{30\sqrt{3}}\sqrt{X}\log\frac{R_2}{R_1}+o(\sqrt{X}),
	\end{align}
	where $C$ is as in Theorem~\ref{thm:C}.
\end{theorem}

\begin{proof}[Proof of Theorem~\ref{thm:fieldasymptotics}]
	For Type~I fields, the discriminant is $-3^3a^2b^2$, so, in the notation of Theorem~\ref{prop:TypeIIabcount} below, we take $N=\frac{1}{3}\sqrt{X/3}$ and get that
	\[
		\N_{\mathrm{I}}(X,R_1,R_2)=\frac{1}{2}\left(\mc{S}_{\mathrm{I}}\left(\frac{1}{3}\sqrt{X/3}, R_2\right)-\mc{S}_{\mathrm{II}}\left(\frac{1}{3}\sqrt{X/3}, R_1\right)\right).
	\]
	The desired result follows from Theorem~\ref{prop:TypeIIabcount}.
	
	Similarly, for Type~II fields, the discriminant is $-3a^2b^2$, so $N=\sqrt{X/3}$, and we have that
	\[
		\N_{\mathrm{II}}(X,R_1,R_2)=\frac{1}{2}\left(\mc{S}_{\mathrm{II}}(\sqrt{X/3}, R_2)-\mc{S}_{\mathrm{II}}(\sqrt{X/3}, R_1)\right),
	\]
	which gives the desired result by Theorem~\ref{prop:TypeIIabcount}.
\end{proof}

The next section contains all the counting results used in the above proof.

\subsection{Counting strongly carefree couples in a hyperbolic pie slice, with congruence conditions}\label{sec:carefree}
A pair $(a,b)$ of positive integers is called a \emph{strongly carefree couple} if $a$ and $b$ are relatively prime and squarefree. In \cite{Carefree}, Pieter Moree counts the number of strongly carefree couples in a box of side $N$ obtaining the asymptotic $CN^2+O(N^{3/2})$, where $C$ is as in Theorem~\ref{thm:C}. Counting pure cubic fields of bounded discriminant and ratio amounts to counting strongly carefree couples below the hyperbola $ab=N$ and within the cone $R^{-1}\leq a/b\leq R$, while counting only those of Type~II imposes a congruence condition modulo 9. In this section, we determine asymptotics for these counts following the methods of \cite{Carefree} and \cite[\S6.2]{Manjul-Ari}.

For $N,R\geq1$, let
\[
	\mc{S}(N,R)=\#\left\{(a,b)\in\ZZ_{\geq1}^2:(a,b)\text{ strongly carefree}, ab\leq N, \frac{1}{R}\leq\frac{a}{b}\leq R\right\},
\]
\[
	\mc{S}_{\mathrm{II}}(N,R)=\#\left\{(a,b)\in\mc{S}(N,R):3\nmid ab, a^2\equiv b^2\mod{9}\right\},
\]
and
\[
	\mc{S}_{\mathrm{I}}(N,R)=\mc{S}(N,R)-\mc{S}_{\mathrm{II}}(N,R).
\]

\begin{theorem}\label{prop:TypeIIabcount}
	For fixed $R$, we have
	\begin{align}
		\mc{S}_{\mathrm{I}}(N,R)&=\frac{4C}{5}N\log R+o(N),\\
		\mc{S}_{\mathrm{II}}(N,R)&=\frac{C}{5}N\log R+o(N),\\
		\mc{S}(N,R)&=CN\log R+o(N).
	\end{align}
\end{theorem}
\begin{proof}
	For the two latter counts, we cover the hyperbolic pie slice with the following four pieces:
	\begin{itemize}
		\item[(i)] $1\leq a\leq\sqrt{RN}, 1\leq b\leq\frac{N}{a}$,
		\item[(ii)] $1\leq a\leq\sqrt{\frac{N}{R}}, 1\leq b\leq\frac{N}{a}$,
		\item[(iii)] $1\leq a\leq\sqrt{\frac{N}{R}}, 1\leq b\leq Ra$,
		\item[(iv)] $1\leq a\leq\sqrt{RN}, 1\leq b\leq\frac{1}{R}a$.
	\end{itemize}
	Letting $\mc{S}^?(N,R)$ denote the number of points with condition $?\in\{\text{(i), (ii), (iii), (iv)}\}$ imposed, and similarly for $\mc{S}^?_{\mathrm{II}}(N,R)$, we have that
	\[
		\mc{S}_?(N,R)=\mc{S}_?^{(i)}(N,R)-\mc{S}_?^{(ii)}(N,R)+\mc{S}_?^{(iii)}(N,R)-\mc{S}_?^{(iv)}(N,R),
	\]
	for $?=\mathrm{II}$ or nothing. For $\mc{S}(N,R)$, we will apply Lemmas~\ref{lem:regionsi-iii} and \ref{lem:regionsii-iv} with $n=1$, whereas for $\mc{S}_{\mathrm{II}}(N,R)$ we take $n=9$ and sum over two $\psi$s, one sending $a$ to $a\mod{9}$, the other to $-a\mod{9}$. Note that all but one of the terms in Lemma~\ref{lem:regionsi-iii} cancel in $\mc{S}_?^{(i)}(N,R)-\mc{S}_?^{(ii)}(N,R)$ since they are independent of $R$. Applying Lemma~\ref{lem:regionsi-iii} with $\rho=R$, then $\rho=1/R$ and noting that Lemma~\ref{lem:regionsii-iv} implies the contribution of $\mc{S}_?^{(iii)}(N,R)-\mc{S}_?^{(iv)}(N,R)$ is negligible, we obtain the stated results.
\end{proof}

The point counting in regions of the form (i)--(iv) above are contained in the following two lemmas.
\begin{lemma}\label{lem:regionsi-iii}
Let $n$ be a positive integer, let $\psi$ be any function from the positive integers prime to $n$ to $(\ZZ/n\ZZ)^\times$, and let $\rho$ be a positive real number. Then,
\begin{align*}
	\sum_{\substack{a\leq\sqrt{\rho N}\\ (a,n)=1}}\mu^2(a)\!\!\!\!\!\sum_{\substack{b\leq N/a\\ (a,b)=1\\ b\equiv\psi(a)\mod{n}}}\!\!\!\!\!\mu^2(b)&=\frac{CN}{n}\prod_{p\mid n}\left(\frac{p^2}{(p-1)(p+2)}\right)\\
	&\cdot\left(\frac{\log N}{2}+\frac{\log \rho}{2}+\gamma+3\kappa+\sum_{p\mid n}\frac{\log p}{p+2}\right)+O(N^{3/4+\epsilon})
\end{align*}
for all $\epsilon>0$,where $\gamma$ is the Euler--Mascheroni constant, $C$ is as in Theorem~\ref{thm:C}, and
\[
	\kappa=\sum_p\frac{\log(p)}{p^2+p-2}.
\]
\end{lemma}
\begin{proof}
	Applying Lemma~\ref{lem:bsum} below to the inner sum with $a^\prime=\psi(a)$ yields
	\[
		\sum_{\substack{b\leq N/a\\ (a,b)=1\\ b\equiv\psi(a)\mod{n}}}\mu^2(b)=\frac{\phi(a)N}{a^2n}\cdot\frac{1}{\zeta(2)}\cdot\prod_{p\mid an}\frac{p^2}{p^2-1}+O(2^{\omega(a)}\sqrt{N/a}).
	\]
	Estimating the error using that\footnote{Indeed, $2^{\omega(m)}\leq d(m)$ and $d(m)=o(m^{\epsilon})$ for all $\epsilon>0$ (see e.g. \cite[Exercise~13.13]{Apostol}).}
	\[
		\sum_{m\leq x}\mu^2(m)\frac{2^{\omega(m)}}{\sqrt{m}}=O(x^{1/2+\epsilon}),\text{ for all }\epsilon>0,
	\]
	we obtain, for all $\epsilon>0$, and up to $O(N^{3/4+\epsilon})$,
	\begin{align*}
		\sum_{\substack{a\leq\sqrt{\rho N}\\ (a,n)=1}}\mu^2(a)\!\!\!\!\!\sum_{\substack{b\leq N/a\\ (a,b)=1\\ b\equiv\psi(a)\mod{n}}}\!\!\!\!\!\mu^2(b)&=\frac{N}{n\zeta(2)}\prod_{p\mid n}\frac{p^2}{p^2-1}\sum_{\substack{a\leq\sqrt{\rho N}\\ (a,n)=1}}\mu^2(a)\frac{\phi(a)}{a^2}\prod_{p\mid a}\frac{p^2}{p^2-1}\\
		&=\frac{N}{n\zeta(2)}\prod_{p\mid n}\frac{p^2}{p^2-1}\sum_{\substack{a\leq\sqrt{\rho N}\\ (a,n)=1}}\mu^2(a)\prod_{p\mid a}\frac{1}{p+1}.
	\end{align*}
	Applying Perron's formula as in Lemma~\ref{lem:perron} below with $k=0$ and $x=\sqrt{\rho N}$ yields the desired result.
\end{proof}

\begin{lemma}\label{lem:regionsii-iv}
	Let $n,\psi$, and $\rho$ be as in Lemma~\ref{lem:regionsi-iii}. Then.
	\[
		\sum_{\substack{a\leq\sqrt{\rho N}\\ (a,n)=1}}\mu^2(a)\sum_{\substack{b\leq a/\rho\\ (a,b)=1\\ b\equiv\psi(a)\mod{n}}}\mu^2(b)=O(N^{3/4+\epsilon}),\text{ for all }\epsilon>0.
	\]
\end{lemma}
\begin{proof}
	This proof proceeds along the same lines as the previous one. Lemma~\ref{lem:bsum} gives
	\[
		\sum_{\substack{b\leq a/\rho\\ (a,b)=1\\ b\equiv\psi(a)\mod{n}}}\mu^2(b)=\frac{\phi(a)}{\rho an}\cdot\frac{1}{\zeta(2)}\cdot\prod_{p\mid an}\frac{p^2}{p^2-1}+O(2^{\omega(a)}\sqrt{a/\rho}).
	\]
	%An old result of Mertens \cite{Mertens{ (see also \cite[\SII.29]{Handbook}) says that\footnote{In fact, Mertens obtained the two main terms and a power-saving error term.}
	%\[
	%	\sum_{m\leq x}d^\ast(m)=O(x\log(x)),
	%\]
	%where $d^\ast(m)$ denotes the number of \emph{unitary} divisors of $m$ (i.e.\ $d\mid m$ such that $(d,m/d)=1$). For squarefree $a$ 
	Again using that $d(m)=o(m^\epsilon)$, for all $\epsilon>0$, we see that\footnote{Here, we simply replace the $\sqrt{a}$ in the error term above with $N^{1/4}$.}
	\[
		\sum_{m\leq x}2^{\omega(m)}=O(x^{1+\epsilon}),\text{ for all }\epsilon>0.
	\]
	This gives, for all $\epsilon>0$, and up to $O(N^{3/4+\epsilon})$,
	\begin{align*}
		\sum_{\substack{a\leq\sqrt{\rho N}\\ (a,n)=1}}\mu^2(a)\!\!\!\!\!\sum_{\substack{b\leq a/\rho\\ (a,b)=1\\ b\equiv\psi(a)\mod{n}}}\!\!\!\!\!\mu^2(b)&=\frac{1}{\rho n\zeta(2)}\prod_{p\mid n}\frac{p^2}{p^2-1}\sum_{\substack{a\leq\sqrt{\rho N}\\ (a,n)=1}}\mu^2(a)\frac{\phi(a)}{a}\prod_{p\mid a}\frac{p^2}{p^2-1}\\
		&=\frac{1}{\rho n\zeta(2)}\prod_{p\mid n}\frac{p^2}{p^2-1}\sum_{\substack{a\leq\sqrt{\rho N}\\ (a,n)=1}}\mu^2(a)\prod_{p\mid a}\frac{p}{p+1}.
	\end{align*}
	%Applying Lemma~\ref{lem:perron} with $k=1$ now and $x=\sqrt{\rho N}$ yields the desired result.
	Since all the terms in the sum are less than 1, the sum itself is $O(\sqrt{N})$, and we are done.
\end{proof}

In order to obtain the counts above with congruence conditions, we need the following improvements to \cite[\S2]{Carefree}.
\begin{lemma}
	Fix $n\in\ZZ\geq1$. For $a$ and $a^\prime$ relatively prime to $n$, let
	\[
		T_{a,a^\prime,n}(x)=\#\{b\leq x:(a,b)=1, b\equiv a^\prime\mod{n}\}.
	\]
	Then
	\[
		T_{a,a^\prime,n}(x)=\frac{\phi(a)x}{an}+O(2^{\omega(a)}),
	\]
	where $\phi$ is the Euler totient function and $\omega(a)$ is the number of distinct prime divisors of $a$.
\end{lemma}
Note that the right hand side is independent of $a^\prime$.
\begin{proof}
Using the identity
\[
	\sum_{d|m}\mu(d)=\begin{cases}
						1&\text{if }m=1\\
						0&\text{otherwise}
					\end{cases}
\]
and the orthogonality of Dirichlet characters in the form
\[
	\frac{1}{\phi(n)}\sum_{\chi\text{ mod }n}\chi(m)\ol{\chi}(m^\prime)=\begin{cases}
						1&\text{if }m\equiv m^\prime\mod{n}\\
						0&\text{otherwise}
					\end{cases}
\]
(where the sum is over all Dirichlet characters modulo $n$), we obtain
\begin{align*}
	T_{a,a^\prime,n}(x)&=\sum_{\substack{b\leq x\\(a,b)=1\\b\equiv a^\prime\mod{n}}}1\\
	&=\sum_{\substack{b\leq x\\b\equiv a^\prime\mod{n}}}\sum_{d\mid\gcd(a,b)}\mu(d)\\
	&=\sum_{b\leq x}\left(\sum_{d\mid\gcd(a,b)}\mu(d)\frac{1}{\phi(n)}\sum_{\chi\text{ mod }n}\chi(a^\prime)\ol{\chi}(b)\right)\\
	&=\frac{1}{\phi(n)}\sum_{d\mid a}\mu(d)\left(\sum_{c\text{ mod }n}\sum_{\chi\text{ mod }n}\chi(a^\prime)\ol{\chi}(c)\right)\\
	&\phantom{=}\cdot\#\{\text{multiples }b\text{ of }d:b\leq x, b\equiv c\mod{n}\}\\
	&=\sum_{d\mid a}\mu(d)\left(\frac{x}{dn}+O(1)\right)\\
	&=\frac{\phi(a)x}{an}+O(2^{\omega(a)}),
\end{align*}
where we've used that
\[
	\sum_{d\mid m}\mu(d)=\frac{\phi(m)}{m}.
\]
\end{proof}

\begin{lemma}\label{lem:bsum}
	Fix $n\in\ZZ\geq1$. For $a$ and $a^\prime$ relatively prime to $n$, let
	\[
		S_{a,a^\prime,n}(x)=\#\{b\leq x:(a,b)=1, b\equiv a^\prime\mod{n},b\text{ squarefree}\}.
	\]
	Then
	\[
		S_{a,a^\prime,n}(x)=\frac{\phi(a)x}{an}\cdot\frac{1}{\zeta(2)}\cdot\prod_{p\mid an}\frac{p^2}{p^2-1}+O(2^{\omega(a)}\sqrt{x}),
	\]
	where $\zeta(2)=\pi^2/6$.
\end{lemma}
Once again, note that the right hand side is independent of $a^\prime$.
\begin{proof}
%here%%%%%%%%%%%%%%
By the inclusion-exclusion principle,
\begin{align*}
	S_{a,a^\prime,n}(x)&=\sum_{\substack{m\leq\sqrt{x}\\(m,an)=1}}\mu(m)T_{a,a^\prime m^{-2},n}(x/m^2)\\
	&=\frac{\phi(a)x}{an}\left(\sum_{\substack{m\leq\sqrt{x}\\(m,an)=1}}\frac{\mu(m)}{m^2}\right)+O(2^{\omega(a)}\sqrt{x})
	\end{align*}
\begin{align*}
	&=\frac{\phi(a)x}{an}\left(\sum^\infty_{\substack{m=1\\(m,an)=1}}\frac{\mu(m)}{m^2}\right)+O(2^{\omega(a)}\sqrt{x})\\
	&=\frac{\phi(a)x}{an}\cdot\frac{1}{\zeta^{(an)}(2)}+O(2^{\omega(a)}\sqrt{x})\\
	&=\frac{\phi(a)x}{an}\cdot\frac{1}{\zeta(2)}\cdot\prod_{p\mid an}\frac{p^2}{p^2-1}+O(2^{\omega(a)}\sqrt{x}).
\end{align*}
\end{proof}

Finally, our application of Perron's formula is dealt with in the following lemma. This approach appears in \cite[\S6.2]{Manjul-Ari}, though we include further details and a slightly more general result.
\begin{lemma}\label{lem:perron}
	Let $k$ be a nonnegative integer and let $n$ be a positive integer. Let
	\[
		A_k(a)=\mu^2(a)\prod_{p\mid a}\frac{p^k}{p+1}\quad\text{and}\quad	f_k(s)=\sum_{\substack{a\geq1\\(a,n)=1}}\frac{A_k(a)}{a^s}.
	\]
	Then,
	\[
		\sum_{\substack{a\leq x\\(a,n)=1}}A_k(a)=\zeta(2)C\prod_{p\mid n}\left(1+\frac{1}{p+1}\right)^{\!\!-1}\cdot\begin{cases}
																				\ds\log(x)+\gamma+3\kappa+\sum_{p\mid n}\frac{\log p}{p+2}&\text{if }k=0,\\
																				\ds\frac{x^k}{k}&\text{if }k>0,
																			\end{cases}
	\]
	where $C$ is as in Theorem~\ref{thm:C} and $\kappa$ is as in Lemma~\ref{lem:regionsi-iii}.
\end{lemma}
\begin{proof}
	Note that $A_k(a)=a^kA_0(k)$ so that $f_k(s)=f_0(s-k)$. Since $A_0(a)\leq 1/a$, $f_0(s)$ converges absolutely for $\Re(s)>0$. Perron's formula (see e.g.\ \cite[Theorem~11.18]{Apostol}) then states that for any $\epsilon>0$ and any $k\geq0$
%	\begin{align*}
%		\sum_{a\leq x}\!{\vphantom{\sum}}^\ast A_k(a)&=\frac{1}{2\pi i}\int_{k+\epsilon-i\infty}^{k+\epsilon+i\infty}f_k(s)\frac{x^s}{s}ds\\
%		&=\frac{1}{2\pi i}\int_{k+\epsilon-i\infty}^{k+\epsilon+i\infty}f_0(s-k)\frac{x^s}{s}ds,
%	\end{align*}
	\[
		\sum_{a\leq x}\!{\vphantom{\sum}}^\ast A_k(a)=\frac{1}{2\pi i}\int_{k+\epsilon-i\infty}^{k+\epsilon+i\infty}f_k(s)\frac{x^s}{s}ds=\frac{1}{2\pi i}\int_{k+\epsilon-i\infty}^{k+\epsilon+i\infty}f_0(s-k)\frac{x^s}{s}ds,
	\]
	where the asterisk on the sum means that if $x$ is an integer, then the last term must be divided by 2. Following \cite[\S6.2]{Manjul-Ari}, we let $h(s)=f_0(s)/\zeta(s+1)$ and we compute
	\[
		\frac{1}{2\pi i}\int_{k+\epsilon-i\infty}^{k+\epsilon+i\infty}h(s-k)\zeta(s-k+1)\frac{x^s}{s}ds
	\]
	by shifting the integral just to the left of $\Re(s)=k$ picking up a pole of $\zeta(s)$ at $s=1$ and, when $k=0$, also the pole of $1/s$ at $s=0$. Since $h(s)$ converges for $\Re(s)>-1/2$, we may use Laurent series to determine the residues that occur. For $k>0$, near $s=k$
	\[
		h(s-k)\zeta(s-k+1)\frac{x^s}{s}=(h(0)+\cdots)\cdot(\frac{1}{s-k}+\cdots)\cdot(x^k+\cdots)\cdot(1/k+\cdots),
	\]
	so the residue is simply $h(0)x^k/k$. For $k=0$, near $s=0$
	\[
		h(s)\zeta(s+1)\frac{x^s}{s}=(h(0)+h^\prime(0)s+\cdots)\cdot(\frac{1}{s}+\gamma+\cdots)\cdot(1+s\log x+\cdots)\cdot1/s,
	\]
	so the residue is $h(0)\gamma+h(0)\log(x)+h^\prime(0)$.

	In order to evaluate $h(0)$, we consider $h(s)/\zeta(s+2)=f(s)/(\zeta(s+1)\zeta(s+2))$ at $s=0$, where its Euler product converges. Note that
	\[
		\left(1+\frac{1}{p+1}\right)(1-p^{-1})(1-p^{-2})=1-\frac{3}{p^2}+\frac{2}{p^3}
	\]
	so
	\[
		\frac{h(0)}{\zeta(2)}=C\prod_{p\mid n}\left(1+\frac{1}{p+1}\right)^{-1}.
	\]
	In order to evaluate $h^\prime(0)$, we instead evaluate $h^\prime(0)/h(0)$. One may verify that
	\[
		h(s)=\prod_{p\mid n}\left(1+\frac{1}{p+1}p^{-s}\right)^{-1}\prod_p\left(1-\frac{1}{p(p+1)}(p^{-s}+p^{-2s})\right).
	\]
	Taking the logarithm, then the derivative, and evaluating at $s=0$ yields
	\[
		\frac{h^\prime(0)}{h(0)}=3\kappa+\sum_{p\mid n}\frac{\log(p)}{p+2}.
	\]
\end{proof}

\subsection*{Acknowledgments}
The author would like to thank Piper Harron, Ari Shnidman, and Rufus Willett for some helpful conversations, as well as the referee for suggesting some points that should be clarified and asking some good questions.

\bibliographystyle{amsalpha}
\bibliography{cubic}

\end{document}